\documentclass[a4paper,10pt, oneside,preprint]{elsarticle}
\usepackage{amsmath,amsthm}
\usepackage{amssymb}
\usepackage{latexsym}
\usepackage{enumerate}
\usepackage{color}
\usepackage[normalem]{ulem}

\newtheorem{thm}{Theorem}[section]
\newtheorem{lem}[thm]{Lemma}
\newtheorem{cor}[thm]{Corollary}
\newdefinition{rmk}[thm]{Remark}
\newproof{pf}{Proof}
\newproof{pot}{Proof of Theorem \ref{thm2}}

\numberwithin{equation}{section}
\newtheorem{prop}[thm]{Proposition}

\theoremstyle{definition}
\newtheorem{defin}[thm]{Definition}

\newcommand{\ca}{\ensuremath{\mathcal A}}

\DeclareMathOperator{\rank}{rank}

\DeclareMathOperator{\N}{\mathbb{N}}
\DeclareMathOperator{\R}{\mathbb{R}}
\DeclareMathOperator{\Z}{\mathbb{Z}}
\DeclareMathOperator{\T}{\mathbb{T}}

\begin{document}

\begin{frontmatter}
\title{Optimal approximation of multivariate periodic Sobolev functions in the sup-norm}
\author{Fernando Cobos\fnref{label2}}
\address{Departamento de An\'{a}lisis Matem\'{a}tico, Facultad de Matem\'{a}ticas, Plaza de las Ciencias 3, 28040 Madrid. Spain.}
\ead{cobos@mat.ucm.es}
\author{Thomas K\"{u}hn \corref{cor1} \fnref{label2}}
\address{Mathematisches Institut, Universit\"at Leipzig, Augustusplatz 10, 04109 Leipzig. Germany.}
\ead{kuehn@math.uni-leipzig.de}
\author{Winfried Sickel}
\address{Mathematisches Institut, Friedrich-Schiller-Universit\"at Jena, Ernst-Abbe-Platz 2, 07737 Jena. Germany.}
\ead{winfried.sickel@uni-jena.de}
\cortext[cor1]{Corresponding author.}

\fntext[label2]{Supported in part by the
Spanish Ministerio de Econom\'ia y Competitividad (MTM2013-42220-P).
}
\begin{abstract}
Using tools from the theory of operator ideals and $s$-numbers, we develop a general approach to transfer estimates for 
$L_2$-approximation of Sobolev functions into estimates for 
$L_\infty$-approximation, with precise control of all involved constants. As an illustration, 
we derive some results for periodic isotropic Sobolev spaces $H^s (\mathbb{T}^d)$ and 
Sobolev spaces of dominating mixed smoothness $H^s _{\rm{mix}} (\mathbb{T}^d)$, always equipped with natural norms.
Some results for isotropic as well as dominating mixed Besov spaces are also obtained.
\end{abstract}
\end{frontmatter}


\section{Introduction}


Nowadays there is an increasing interest in the study of linear approximation (approximation numbers) in the context of 
Sobolev spaces. 
This is motivated by the fact that a number of problems in finance and quantum chemistry are modeled on function spaces on
high-dimensional domains, and often the functions to be approximated have some Sobolev regularity (see, for example, \cite{Y}).

In the papers \cite{KSU1, KSU2}, T.~Ullrich and two of the present authors studied linear approximation of functions
in the isotropic periodic Sobolev spaces $H^s (\mathbb{T}^d)$ and in the smaller spaces $H^s _{\rm{mix}} (\mathbb{T}^d)$ 
of dominating mixed smoothness, see also the paper by D\~ung and Ullrich \cite{DiUl13} in this context (but these authors used different norms). 
The error was measured in $L_2 (\mathbb{T}^d)$, with particular emphasis on the constants and 
their dependence on the dimension $d$. In \cite{KSU1, KSU2} the exact decay rate of the constants as $d \rightarrow \infty$ 
was found, which turned out to be polynomial in $d$ for the isotropic spaces, and superexponential in $d$ for the mixed spaces.

In the present paper we deal with the same approximation problems, but now the error is measured in the $\sup$-norm. 
For this aim, we develop a general method which allows to transfer results on $L_2$-approximation into results on 
$L_\infty$-approximation. Our approach is based on tools from the theory of operator ideals and s-numbers, in the 
sense of the monographs by Pietsch \cite{P1, P3}. In particular, we work with absolutely $2$-summing operators.

We consider mainly spaces of periodic functions on the $d$-dimensional torus, whose
norms are weighted $\ell_2$-sums of Fourier coefficients. In addition to the correct rate of decay 
of the approximation numbers, we also obtain very precise information on the "hidden" constants, especially their 
dependence on the dimension $d$. As an illustration, we apply our general method to isotropic Sobolev
spaces $H^s (\mathbb{T}^d)$ and Sobolev spaces $H^s _{\rm{mix}} (\mathbb{T}^d)$ of dominating mixed smoothness. 
Our estimates allow to control the dependence of the constants on the dimension $d$, the smoothness $s$ and the particular norm used in the Sobolev space.
We also show that $L_\infty (\mathbb{T}^d)$ can be replaced in the isotropic case by the Wiener algebra $\ca (\mathbb{T}^d)$ or
the Besov space $B^0_{\infty,1} (\T^d)$, and in the mixed case by the 
Besov space of dominating mixed smoothness $S^0_{\infty,1} B(\T^d)$, without changing the associated approximation numbers. 
This is not only surprising from a theoretical point of view, but also of some practical use, 
since the Littlewood-Paley characterization of the target spaces simplifies the computation of approximation numbers
(or any other $s$-numbers).

Sobolev and Besov spaces of dominating mixed smoothness are representing first attempts in 
approximation theory to deal with high dimensions.
These spaces are much smaller than their isotropic counterparts, they have 
attracted a lot of interest in approximation theory since the early 1960s
(mainly in Russia), and also in 
IBC (information-based complexity), we refer to the monographs \cite{Tem} and 
\cite{NoWo08}, \cite{NoWo10}, \cite{NoWo12}. 

Finally, some upper estimates for approximation in the $L_p$-norm, $2 <p<\infty$, are established.

The plan of the paper is simple. In Section $2$ we recall the basic notions on function spaces and operator theory 
that we shall need. Then, in Section $3$, we establish the abstract results, and 
in the final Section $4$ we apply these results to Sobolev and Besov spaces.


\section{Preliminaries}\label{Section 2}


First we fix some notation. 
For $d\in\mathbb{N}$ , $x=(x_1 , \ldots, x_d) \in \mathbb{R}^d $ and  $0<p< \infty$
let $|x|_p =\Big ( \sum_{j=1} ^d |x_j|^p\Big )^{1/p}$ ,  and for $p=\infty$ we set
$|x|_\infty =\max_{1\leq j \leq d}|x_j|$. 
In what follows $\T$ denotes the torus, i.e. $\T =[0, 2\pi]$ where the endpoints 
of the interval are identified, and $\mathbb{T}^d$ stands for the $d$-dimensional torus. 
We equip $\mathbb{T}^d$ with the \emph{normalized}
Lebesgue measure $(2\pi)^{-d}dx$. Consequently,  $\{e^{i k x}: k\in \mathbb{Z}^d\}$ is an
orthonormal basis in $L_2 (\mathbb{T}^d)$, where $kx=\sum_{j=1}^d k_jx_j$\,.

The \emph{Fourier coefficients} of a function $f\in L_1 (\mathbb{T}^d)$ are defined as
\[
\widehat{f}(k)= (2\pi)^{-d}\int_{\mathbb{T}^d} f(x) e^{- i k x} dx\quad,\quad k\in \mathbb{Z}^d\,.
\]
For $0<s<\infty$ and $0< r \leq\infty$ we denote by $H^{s,r}(\mathbb{T}^d)$ 
the \emph{isotropic Sobolev space} formed by all $f\in L_2(\mathbb{T}^d)$ having a finite norm
\[
\|f|H^{s,r}(\mathbb{T}^d)\| = 
\Big (\sum_{k\in \mathbb{Z}^d}\big(1+\sum_{j=1}^d |k_j|^r \big)^{2s/r} |\widehat{f}(k)|^2 \Big )^{1/2} \,.
\]
Clearly, for fixed $s$, all these norms are equivalent, whence all spaces $H^{s,r}(\mathbb{T}^d)$ with $0<r\le\infty$ 
coincide. The superscript $r$ just indicates which norm we are considering.

For \emph{integer smoothness} $s=m\in \mathbb{N}$, the most natural norms are those with $r=2$ and $r=2m$. 
Indeed, let $D^\alpha f$ be the distributional derivative of 
$f$ of order $\alpha=(\alpha_1, \dots, \alpha_d)$. As shown in \cite{KSU1}, one has
\[
\frac{1}{\sqrt{m!}}\|f|H^{m,2}(\mathbb{T}^d)\| \le \Big ( \sum_{|\alpha|_1 \leq m} 
\big\|D^\alpha f | L_2 (\mathbb{T}^d)\big\|^2 \Big )^{1/2}
\le \|f|H^{m,2}(\mathbb{T}^d)\|\,.
\]
Note that the equivalence constants depend only on the smoothness $m$, but not on the dimension $d$. 

If $r=2m$, one has even equality
\[
\|f|H^{m,2m}(\mathbb{T}^d)\| = \Big ( \|f|L_2(\mathbb{T}^d )\|^2 + \sum_{j=1} ^d  
\big\|\frac{\partial ^m f}{\partial x_j ^m } \big| L_2 (\mathbb{T}^d)\big\|^2 \Big )^{1/2}\,.
\]

The Sobolev space $H^{s,r} _{\rm{mix}}(\mathbb{T}^d)$ of \emph{dominating mixed smoothness} consists of all 
$f\in L_2(\mathbb{T}^d)$ having a finite norm
\[
\|f|H^{s,r}_{\rm{mix}}(\mathbb{T}^d)\| = \Big ( \sum_{k\in \mathbb{Z}^d}\prod _{j=1} ^d \big(1+ |k_j|^r \big)^{2s/r} 
|\widehat{f}(k)|^2 \Big )^{1/2} 
\]
(see \cite{KSU2}). For $s=m\in \mathbb{N}$ and $r=2m$, it turns out that
\[
\|f|H^{m,2m}_{\rm{mix}}(\mathbb{T}^d)\| =\Big ( \sum_{\alpha \in \{0,m\}^d} \|D^\alpha f | 
L_2 (\mathbb{T}^d)\|^2 \Big )^{1/2}\,.
\]

Let us give a short comment on the role played by the parameter $r$.  
Of course, for fixed $d$ and $s$, different $r$'s result in equivalent norms, but the equivalence constants 
depend heavily on $d$. From our point of view it is interesting to see how these changes of the norm 
influence the behavior of the associated approximation numbers. 


The $n$-th \emph{approximation number} of a (bounded linear) 
operator $T: X \to Y$ between Banach spaces is defined as
\[a_n (T)= \inf \{\|T-A\| : \rank A<n\}\, ,\]
i.e. $a_n(T)$ is the optimal error of approximating $T$ by operators of rank less than $n$.
We refer to \cite{P1, P2, P3} for properties of these numbers. 
Let us just recall that for compact operators $T$ between Hilbert spaces $a_n(T)$ 
coincides with the $n$-th singular number $s_n (T)$ of $T$.

An operator $T: X\to Y$ is called \emph{absolutely $2$-summing} if there is a constant $C>0$ such that for 
all $n\in \mathbb{N}$ and all $x_1 , \dots, x_n \in X$ the inequality
\begin{equation}\label{F2}
\Big (\sum_{j=1} ^n \|\, T x_j \, |Y\|^2 \Big )^{1/2} \leq C \sup_{\|x'|X'\|\leq 1}
\Big( \sum_{j=1} ^n |\langle x_j , x' \rangle|^2\Big)^{1/2}\,.
\end{equation}
holds, where $X'$ is the dual space of $X$. The $2$-summing norm $\pi _2 (T)$ is
defined as the infimum of all $C>0$ satisfying (\ref{F2}). For more information we refer to \cite[Chapter 17]{P1}.
Later on we shall use the fact that for operators $T:H\to G$ between Hilbert spaces $H$ and $G$, the $2$-summing norm is 
equal to the Hilbert-Schmidt norm, whence for any orthonormal basis $\{e_i: i\in I\}$ of $H$ it holds
$$
\pi_2(T:H\to G)=\Big(\sum_{i\in I}\Vert Te_i|G\Vert^2\Big)^{1/2}\,.
$$
Given two sequences  $(a_n)$ and $(b_n)$, we write $a_n\lesssim b_n$ if
there is a constant $c>0$ such that $a_n \leq c\, b_n$ for all $n\in\mathbb{N}$. 
The \emph{weak equivalence} $a_n \thicksim b_n$ means that
$a_n\lesssim b_n$ and $a_n\lesssim b_n$, and the \emph{strong equivalence} $a_n \asymp b_n$ means
$\lim\limits_{n\to\infty}\frac{a_n}{b_n}=1$.

\section{General results}\label{Section 3}


In what follows $F_d(w)$ always stands for a Hilbert space of integrable functions on the $d$-dimensional torus 
$\mathbb{T}^d$ such that
\begin{equation}\label{ws-06}
f\in F_d(w) \enspace\Longleftrightarrow\enspace \|f|F_d(w)\|:= 
\Big ( \sum_{k\in \mathbb{Z}^d}w(k)^2 |\widehat{f}(k)|^2 \Big )^{1/2} < \infty\,.
\end{equation}
Here $w(k)>0\,,\, k\in \mathbb{Z}^d \,,$ are certain weights. Important examples of such spaces are the Sobolev spaces 
$H^{s,r}(\mathbb{T}^d)$ and $H^{s,r} _{\rm{mix}}(\mathbb{T}^d)$ introduced in the previous section.

We shall also deal with the \emph{Wiener algebra} $\ca(\mathbb{T}^d)$, which is the collection of all integrable 
functions on $\mathbb{T}^d$ with absolutely convergent Fourier series. $\ca(\mathbb{T}^d)$ is a Banach space 
with respect to the norm
\[\|f|\ca(\mathbb{T}^d)\|=\sum_{k\in \mathbb{Z}^d} |\widehat{f}(k)|\,. \]
Here, as well as in the context of Sobolev and Besov spaces, we shall make use of the following convention:
If the equivalence class of a measurable function $f$ contains a continuous representative, 
then we call $f$ itself continuous and work with the continuous representative.

Several necessary and sufficient conditions are known for a function to belong to $\ca(\mathbb{T}^d)$ 
(see \cite[Chapitre II]{Ka}). We just recall a characterization that describes the Wiener algebra as
an approximation space. Let $O_0=\{0\}$ and, for $n\in \mathbb{N}$, let $O_n$ be the set of all trigonometrical 
polynomials having at most $n$ non-zero coefficients. According to a result of 
Ste$\check{\text{c}}$kin \cite{Ste} (see also \cite{Pi}), a function $f$ belongs to $\ca(\mathbb{T}^d)$ if and only if
\[\sum_{n=1}^\infty n^{-1/2} \inf\{\|f-p|L_2 (\mathbb{T}^d )\|:p\in O_n\} < \infty\,.\]

Our first general result provides a necessary and sufficient condition on the weights $w(k)$ which guarantees 
the existence of continuous or, equivalently, compact embeddings of $F_d(w)$ in 
$\ca(\mathbb{T}^d)$, $C(\mathbb{T}^d)$ or $L_\infty(\mathbb{T}^d)$.

\begin{thm}\label{Theorem 1}
    The following conditions are equivalent.
    \begin{enumerate}
    \item[(i)] $F_d(w) \hookrightarrow  \ca(\mathbb{T}^d)$ compactly
        \item[(ii)]  $F_d(w) \hookrightarrow  \ca(\mathbb{T}^d)$ boundedly
        \item[(iii)]   $F_d(w) \hookrightarrow C(\mathbb{T}^d)$ compactly
        \item[(iv)]  $F_d(w) \hookrightarrow C(\mathbb{T}^d)$ boundedly
        \item[(v)] $F_d(w) \hookrightarrow L_\infty (\mathbb{T}^d)$ compactly
        \item[(vi)]  $F_d(w) \hookrightarrow L_\infty (\mathbb{T}^d)$ boundedly
        \item[(vii)] $\sum\limits_{k\in \mathbb{Z}^d} w(k)^{-2} < \infty$
    \end{enumerate}
\end{thm}
\begin{pf}
Due to the continuous embeddings
$\ca(\mathbb{T}^d) \hookrightarrow C(\mathbb{T}^d)\hookrightarrow L_\infty(\mathbb{T}^d)$, the implications
$(i)\Rightarrow (ii) \Rightarrow(iv) \Rightarrow(vi)$ and $(i) \Rightarrow (iii) \Rightarrow (v) \Rightarrow (vi)$ are trivial.
So it remains to prove $(vi)\Rightarrow (vii)$ and $(vii)\Rightarrow (i)$.\\
{\em Step 1.} Proof of $(vi)\Rightarrow (vii)$:\\ 
The formal identity from $L_\infty (\mathbb{T}^d)$ into $L_2(\mathbb{T}^d)$ is absolutely
$2$-summing, whence $I_d :F_d(w) \to L_2(\mathbb{T}^d)$ is also $2$-summing, and therefore it is
Hilbert-Schmidt, because $F_d(w) $ and $ L_2(\mathbb{T}^d)$ are Hilbert spaces. 
Let $\varphi_k (x)=e^{ i k x}/w(k)$. By definition of $F_d(w)$, the set $\{\varphi_k : k\in \mathbb{Z}^d\}$ 
is a complete orthonormal system in $F_d(w)$, and consequently we have
 \[ 
 \sum_{k\in \mathbb{Z}^d}\frac{1}{w(k)^2}  = \sum_{k\in \mathbb{Z}^d}\|\varphi_k | L_2 (\mathbb{T}^d )\|^2
         = \pi_2 \big(I_d : F_d(w) \to L_2 (\mathbb{T}^d)\big )^2 < \infty. 
 \]
 {\em Step 2.} Proof of $(vii)\Rightarrow (i)$:\\ 
Let $A:F_d(w) \to \ell_2 (\mathbb{Z}^d )$ be the operator defined 
by $Af=(w(k)\widehat{f}(k))_{k\in \mathbb{Z}^d}$. Clearly $A$ is an isometry. 
For $\xi=(\xi_k)_{k\in \mathbb{Z}^d}\,$ we define $D\xi=(\xi_k / w(k))_{k\in \mathbb{Z}^d}$. 
Using (vii) and H\"{o}lder's inequality, we see that $D:\ell_2(\mathbb{Z}^d) \to \ell_1(\mathbb{Z}^d)$ is 
bounded. Now, for $\xi=(\xi_k)\in \ell_1(\mathbb{Z}^d)$, we set 
$(B\xi)(x)=\sum_{k\in \mathbb{Z}} \xi_k e^{i k x}$. The series that defines $B\xi$ converges absolutely and uniformly, 
whence the operator $B:\ell_1(\mathbb{Z}^d)\to  \ca (\mathbb{T}^d)$ is bounded with norm one, and the 
following commutative diagram holds.

\begin{center}
\setlength{\unitlength}{.7mm}
\begin{picture}(100,50)
\linethickness{.5pt} \put(3,10){$\ell_2 (\mathbb{Z}^d)$} \put(3,40){$F_d(w)$}
\put(63,10){$\ell_1 (\mathbb{Z}^d)$} \put(63,40){$\ca (\mathbb{T}^d)$}

\put(18,42){\vector(1,0){44}} \put(12,35){\vector(0,-1){18}}
\put(19,12){\vector(1,0){43}} \put(67,17){\vector(0,1){20}}

\put(68,25){$B$} \put(40,45){$I_d$} \put(7,25){$A$}
\put(40,14){$D$}

\end{picture}
\end{center}
Let $(\sigma_n)_{n\in \mathbb{N}}$ be the non-increasing rearrangement of 
$(1/w(k))_{k\in \mathbb{Z}^d}$, which by assumption belongs to $\ell_2(\mathbb{Z}^d)$,
and consider the associated diagonal operator 
$D_\sigma x=(\sigma_n x_n)_{n\in \mathbb{N}}$ for $x=(x_n)_{n\in \mathbb{N}}$. 
The rearrangement of $(1/w(k))_{k\in \mathbb{Z}^d}$ into the sequence $(\sigma_n)_{n\in\N}$
defines a one-to-one correspondence between the index sets $\mathbb{Z}^d$ and $\mathbb{N}$, 
whence the multiplication property of approximation numbers gives
\[
a_n (D:\ell_2(\mathbb{Z}^d) \to \ell_1(\mathbb{Z}^d)) = a_n (D_\sigma :\ell_2 \to \ell_1) \, . 
\]
The commutative diagram yields
\begin{eqnarray*}   \nonumber
a_n(I_d  : F_d(w) \to \ca(\mathbb{T}^d))
&\leq& \|A\|\cdot a_n (D:\ell_2(\mathbb{Z}^d) \to\ell_1(\mathbb{Z}^d))\cdot \|B\| \\
&=& a_n (D_\sigma :\ell_2 \to \ell_1) =\Big( \sum_{j=n} ^\infty \sigma_j ^2\Big)^{1/2}\,,
\end{eqnarray*}
where the last equality follows from \cite[Theorem 11.11.4]{P1}. This implies \\
$\lim\limits_{n\to \infty} a_n(I_d: F_d(w) \to \ca(\mathbb{T}^d))=0$\,, and thus the embedding 
$I_d: F_d(w)\to \ca(\mathbb{T}^d)$ is compact. The proof is complete.
    \qed
\end{pf}

In Step 2 of the preceding proof we have obtained the following inequality, which we state for later use as a separate lemma.

\begin{lem}\label{thomas2}
Let $F_d(w)$ be given by weights $w(k)$ satisfying $\sum\limits_{k\in \Z^d} w(k)^{-2}<\infty$. 
Then we have for all $n\in\N$
$$
a_n(I_d:F_d(w)\to \ca(\mathbb{T}^d))\le 
\Big(\sum_{j=n}^\infty \sigma_j^2\Big)^{1/2}\,,
$$
where $(\sigma_j)_{j\in\mathbb{N}}$ is the non-increasing rearrangement of $(1/w(k))_{k\in\mathbb{Z}^d}$.
\end{lem}

Next we show a lower estimate for these approximation numbers. Here we allow a greater generality.
Recall that $s_j (T)$ denotes the $j$-th singular number of a compact operator $T$ between Hilbert spaces.

\begin{lem}\label{thomas1}
Let $H$ be a Hilbert space, let $(\Omega,\Sigma,\nu)$ be a finite measure space, and $id_\nu:L_\infty(\nu)\to L_2(\nu)$ 
the formal identity. Then one has, for every bounded linear operator $T:H\to L_\infty(\nu)$ and all $n\in\mathbb{N}$, the 
estimate
$$
a_n(T)\ge \frac{1}{\sqrt{\nu(\Omega)}} \, \Big(\sum_{j=n}^\infty s_j(id_\nu\, T)^2\Big)^{1/2}\,.
$$
\end{lem}
\begin{pf}
Let $A:H\to L_\infty(\nu)$ be an arbitrary operator with $\rank A<n$. By the additivity of singular 
numbers we get for all $j\ge n$ the inequality 
$$
s_j(id_\nu \, T)\le s_n(id_\nu\, A)+s_{j+1-n}(id_\nu \, (T-A))=s_{j+1-n}(id_\nu \, (T-A))\,,
$$
where we took into account that $s_n(id_\nu\, A)=0$. Changing the running index $j\ge n$ to $k=j+1-n\ge 1$, this implies
\begin{align*}
\sum_{j=n}^\infty s_j(id_\nu\, T)^2 &\le
\sum_{k=1}^\infty s_k(id_\nu\, (T-A))^2
=\pi_2(id_\nu\, (T-A))^2\\
&\le \pi_2(id_\nu)^2\cdot\Vert T-A\Vert^2 = \nu(\Omega)\cdot \Vert T-A\Vert^2\,.
\end{align*}
Here we also used the well-known facts that 
$$\pi_2(S)^2=\sum_{k=1}^\infty s_k(S)^2$$ for compact operators $S$ acting between Hilbert spaces, see \cite[Theorem~4.10]{DJT},
and $$\pi_2(id_\nu:L_\infty(\nu)\to L_2(\nu))=\sqrt{\nu(\Omega)}\,,$$
see \cite[Example (d),p.~40]{DJT}.
Passing to the infimum over all operators $A$ of $\rank A<n$, we arrive at the desired inequality
$$
a_n(T)
\ge \frac{1}{\sqrt{\nu(\Omega)}}\Big(\sum_{j=n}^\infty s_j(id_\nu\, T)^2\Big)^{1/2}\,.
$$
\qed\end{pf}
The arguments in the above proof are similar to the ones used by Osipenko and Parfenov in \cite{OP}. 
We thank Heping Wang for pointing out this paper to us. 

As an immediate consequence of the preceding two lemmata 
we obtain the following result.

\begin{thm}\label{Theorem 2}
Let $F_d(w)$ be given by weights satisfying $\sum\limits_{k\in \Z^d} w(k)^{-2}<\infty$, and let
$(\sigma_j)_{j\in\mathbb{N}}$ denote the non-increasing rearrangement of $(1/w(k))_{k\in\Z^d}$.
Moreover, let 
$$
G_d\quad = \quad \ca(\mathbb{T}^d)\quad\text{or}\quad C(\mathbb{T}^d)\quad\text{or}\quad L_\infty(\mathbb{T}^d)\,.
$$
Then one has for all $n\in\mathbb{N}$ 
\begin{equation}\label{2-infty}
a_n(I_d: F_d(w)\to G_d)= \Big(\sum_{j=n}^\infty \sigma_j^2\Big)^{1/2}\,.
\end{equation}
\end{thm}

\begin{pf}
Recall that the torus $\mathbb{T}^d$ is equipped with the \emph{normalized} Lebesgue measure.
In view of the norm one embeddings
$$
\ca (\mathbb{T}^d)\hookrightarrow C(\mathbb{T}^d)\hookrightarrow L_\infty(\mathbb{T}^d)
$$
and the multiplicativity of the approximation numbers, we get from Lemma \ref{thomas1} and Lemma \ref{thomas2}  
the following chain of inequalities
\begin{align*}
\Big(\sum_{j=n}^\infty \sigma_j^2\Big)^{1/2}&\le\, a_n(I_d: F_d(w)\to L_\infty(\mathbb{T}^d))\\
&\le\, a_n(I_d: F_d(w)\to C(\mathbb{T}^d))\\
&\le\, a_n(I_d: F_d(w) \to \ca(\mathbb{T}^d))\le\Big(\sum_{j=n}^\infty \sigma_j^2\Big)^{1/2}\,.
\end{align*}
The proof is finished. 

\qed\end{pf}

\begin{rmk}
Since $\sigma_j=a_j(I_d: F_d(w)\to L_2(\mathbb{T}^d))$, equation (\ref{2-infty})
gives the nice formula
$$
a_n(I_d: F_d(w)\to L_\infty(\T^d))=
\Big(\sum_{j=n}^\infty a_j(I_d: F_d(w)\to L_2(\T^d))^2\Big)^{1/2}\,.
$$
The same is true if we replace $L_\infty(\T^d)$ with $C(\T^d)$ or $\ca(\T^d)$.
\end{rmk}

\section{Applications}\label{Section 4}


In this section we apply our general result Theorem \ref{Theorem 2} to various Sobolev  and Besov spaces. 


\subsection{Embeddings of isotropic Sobolev spaces}\label{Section 4.1}


First we consider the isotropic Sobolev spaces $H^{s,r}(\mathbb{T}^d)$, introduced in Section 2, 
where  $s>0$ and $0<r \leq\infty$. Taking the weights
\[
w(k):=w_{s,r}(k)=\big(1+\sum_{j=1} ^d |k_j|^r \big)^{s/r}\,,
\]
we get $F_d(w)= H^{s,r}(\mathbb{T}^d)$. 
Note that, for fixed $d\in \mathbb{N}$ and $s>0$, all weights $w_{s,r}$ with $0<r\le\infty$ are equivalent.

For the non-increasing rearrangement $(\sigma_n)_{n\in \mathbb{N}}$ of $(1/w_{s,r}(k))_{k\in \mathbb{Z}^d}$ we have
$\sigma_n=a_n(I_d:H^{s,r}(\mathbb{T}^d) \to L_2(\mathbb{T}^d))$.
It is a classical fact that 
\[
H^{s,r}(\mathbb{T}^d) \hookrightarrow C(\mathbb{T}^d) 
\qquad \Longleftrightarrow \qquad s>\frac d2\, .
\]
However, it can be checked directly that 
\begin{equation} \label{F4}
\sum_{k\in \mathbb{Z}^d} \frac{1}{w_{s,r}(k)^2} < \infty 
\qquad \Longleftrightarrow \qquad s>\frac{d}{2} \,.
\end{equation}
By Theorem \ref{Theorem 1} we conclude that the condition $s>d/2$ is necessary and sufficient for the 
existence of an embedding of $H^{s,r}(\mathbb{T}^d)$ into  $\ca (\mathbb{T}^d)$ (or $C (\mathbb{T}^d)$ or
$L_\infty (\mathbb{T}^d)$\,). 
The following result was shown in \cite[Theorems 4.3, 4.11 and 4.14]{KSU1} for the special values $r=1,2,2s$, 
but the proof works for all $0<r\le \infty$.

\begin{prop}\label{Prop 1}
Fix $d\in \mathbb{N}$ and $s>0$. Then, for all $0<r\le \infty$, it holds
\[
\lim_{n \rightarrow \infty} n^{s/d}\, a_n(I_d : H^{s,r}(\mathbb{T}^d) \to L_2 (\mathbb{T}^d))
= \mbox{\rm vol}(B^d_r )^{s/d}\, .
\]
Here $B^d_r$ is the unit ball in $\mathbb{R}^d$ with respect to the (quasi)-norm $|\cdot|_r$. 
\end{prop}

\begin{rmk}
\label{Remark 1} 
\rm
(i) One can rephrase Proposition \ref{Prop 1} as a strong equivalence
$$
a_n(I_d : H^{s,r}(\mathbb{T}^d) \to L_2 (\mathbb{T}^d))
\asymp \mbox{\rm vol}(B^d_r )^{s/d}\, n^{-s/d}.
$$ 
The importance of this result is that it provides asymptotically optimal 
constants, for arbitrary fixed $d,s,r$.
\\
(ii)
We comment on the influence of the parameter $r$. The formula 
\[
\mbox{vol}(B_{r}^d):= \mbox{vol}\Big\{x \in \mathbb{R}^d: \: \sum_{j=1}^d |x_j|^{r}\le 1 \Big\} = 2^d\,
\frac{\Gamma (1+1/r)^d}{\Gamma (1+d/r)}\, ,
\]
for the volume of the unit ball $B_{r}^d$  is well-known, see, e.g.,  Wang \cite{Wa}, where 
$$
    \Gamma(1+x) = \int_{0}^{\infty} t^x e^{-t}\,dt\quad,\quad x>0,
$$
denotes the Gamma function. As shown in \cite{KSU1}, for all $x>0$ one has
$$
\left(\frac x e\right)^x \le \Gamma(1+x) \le (x+1)^x\,,
$$
which implies the two-sided estimate
\[
\frac{2^s}{\big(e(d+r)\big)^{s/r}} \leq \mbox{\rm vol}(B^d_{r})^{s/d} 
\leq \frac{2^s\big(e(r+1)\big)^{s/r}}{d^{s/r}}\,.
\]
Hence, for fixed $s$ and $r$, it holds
\[
\mbox{\rm vol}(B^d_r )^{s/d} \thicksim d^{-s/r}\,  \quad\text{as } d \to \infty \, .
\]
\end{rmk}

\begin{thm}\label{Theorem 3}
Let $d\in \mathbb{N}$ , $s>d/2$ and $0<r\le\infty$. Then
\begin{equation}\label{ws-02}
\lim_{n \rightarrow \infty} n^{s/d-1/2}\, a_n(I_d : H^{s,r}(\mathbb{T}^d) \to L_\infty (\mathbb{T}^d))
= \sqrt{\frac{d}{2s-d}} \, \mbox{\rm vol}(B^d_r )^{s/d} \, .
\end{equation}
\end{thm}

\begin{pf}
By Theorem \ref{Theorem 2} we have 
$$
n^{s/d-1/2} a_n(I_d : H^{s,r}(\mathbb{T}^d) \to L_\infty (\mathbb{T}^d)) 
=  n^{s/d-1/2} \Big(\sum_{j=n}^\infty \sigma_j^2\Big)^{1/2}\,,
$$
where $\sigma_j=a_j(I_d : H^{s,r}(\mathbb{T}^d) \to L_2(\mathbb{T}^d))$.
From Proposition \ref{Prop 1} we know that for any $\varepsilon >0 $ there exists a natural number $n_0(\varepsilon)$
such that 
$$
|n^{s/d} \sigma_n - \mbox{\rm vol}(B^d_r )^{s/d}|
< \varepsilon \qquad \mbox{for all}\quad n \ge n_0 (\varepsilon).
$$
This yields, if $n\ge n_0(\varepsilon)$,
$$
\big(\mbox{\rm vol}(B^d_r )^{s/d} - \varepsilon\big)^2 \sum_{j=n}^\infty j^{-2s/d}
\le \sum_{j=n}^\infty \sigma_j^2
 \le  \big(\mbox{\rm vol}(B^d_r )^{s/d} + \varepsilon\big)^2 \sum_{j=n}^\infty j^{-2s/d}\,.
$$
 Comparing the series with an integral, 
 \[
\frac{n^{1-2s/d}}{2s/d -1} = \int_{n}^\infty \frac{dx}{x^{2s/d}}  \le  
\sum_{j=n}^\infty j^{-2s/d} \le 
\int_{n-1}^\infty \frac{dx}{x^{2s/d}}= \frac{(n-1)^{1-2s/d}}{2s/d -1}\,,
 \]
 we obtain 
\begin{align*}
\big(\mbox{\rm vol}(B^d_r )^{s/d} - \varepsilon\big)^2 \, \frac{d}{2s-d} & \le
n^{{2s/d-1}}\, a_n(I_d : H^{s,r}(\mathbb{T}^d) \to L_\infty (\mathbb{T}^d))^2\\
& \le \big(\mbox{\rm vol}(B^d_r )^{s/d} + \varepsilon\big)^2 \, \frac{d}{2s-d}
  \left(\frac{n}{n-1}\right)^{2s/d -1} \, .
\end{align*}
Since this is true for all $\varepsilon>0$ and sufficiently large $n$, the claim follows 
by letting first $n \to \infty$ and then $\varepsilon\to 0$.
\qed
\end{pf}

\begin{rmk}\label{Remark 2}
\rm
(i)  Rephrasing Theorem \ref{Theorem 3} in terms of the strong equivalences
\begin{align*}
a_n(I_d : H^{s,r}(\mathbb{T}^d) \to L_\infty (\mathbb{T}^d))
&\asymp n^{1/2-s/d}\, \sqrt{\frac{d}{2s-d}} \, \mbox{\rm vol}(B^d_r )^{s/d}\\
 &\asymp n^{1/2}\, \sqrt{\frac{d}{2s-d}} \, a_n(I_d : H^{s,r}(\mathbb{T}^d) \to L_2(\mathbb{T}^d))
\end{align*}
we see that, compared to $L_2$-approximation, the rate of $L_\infty$-approximation is worse by a factor $n^{1/2}$. 
Moreover, for the constant we need the correction factor $\sqrt{\frac{d}{2s-d}}$\,.\\
(ii)
Of course, the asymptotic behavior $a_n(I_d : H^{s,r}(\mathbb{T}^d) \to L_\infty (\mathbb{T}^d))\sim n^{1/2-s/d}$
of the approximation numbers is known since some time, see e.g. the monograph 
by Temlyakov \cite[Theorem~1.4.2, Theorem~2.4.2]{Tem}, but only in the sense of weak equivalence, without explicit constants.
At the end of that book, some historic remarks on the periodic case can be found. For the non-periodic case 
we refer to Edmunds and Triebel \cite{ET} and to Vyb{\'\i}ral \cite{Vy}.
The novelty of Theorem \ref{Theorem 3} is that it gives \emph{strong equivalence} and provides
exact information about the asymptotically optimal constants and their dependence on $d,s$ and $r$.
\\
(iii) Theorem \ref{Theorem 2} yields that relation \eqref{ws-02} remains true if we replace 
$L_\infty (\mathbb{T}^d)$ with $\ca (\mathbb{T}^d)$ or $C(\mathbb{T}^d)$. 
\end{rmk}

After determining  the asymptotic behaviour of  $a_n(I_d : H^{s,r}(\mathbb{T}^d) \to L_\infty (\mathbb{T}^d))$, we study next 
single estimates between $a_n (I_d)$ and $n^{1/2-s/d}$. We need some preparation.

\begin{prop}\label{Corollary}
Assume that $\sigma_n:=a_n (I_d:F_d(w) \to  L_2 (\mathbb{T}^d))$ satisfies
\[
A n^{-\alpha}(\log n)^\beta \leq \sigma_n \leq Bn^{-\alpha}(\log n)^{\beta}\quad \text{for all} \quad n \geq N.
\]
for some $\alpha > 1/2$, $\beta\ge 0$, $0<A \le B < \infty$ and $N\in \mathbb{N}$\,.
Then it follows that
    \begin{enumerate}
    \item[(i)]  $a_{n+1}(I_d:F_d(w) \to  L_\infty (\mathbb{T}^d)) \leq
        B \, \sqrt{\frac{2}{2\alpha -1}} \, n^{1/2 - \alpha} (\log n)^\beta$ 
				
				\hspace{6cm} for all $n\geq \max(N, e^{4\beta /(2\alpha -1)})$
				
 and

 \item[(ii)] $a_n (I_d:F_d(w) \to  L_\infty (\mathbb{T}^d))\geq A \, \sqrt{\frac{1}{4\alpha -2}} \, n^{1/2 - \alpha} (\log 2n)^\beta$ 
    for all  $n\geq N \ge 2$.

    \end{enumerate}

\end{prop}
\begin{pf}
 {\em  Step 1.} Let us show $(i)$. The upper estimate in Theorem \ref{Theorem 2} gives for $n\geq N$
   \begin{equation} \label{(1)}
   a_{n+1} ^2 := 
	a_{n+1}(I_d:F_d(w) \to  L_\infty (\mathbb{T}^d))^2 \leq B^2 \sum_{j=n+1} ^\infty \frac{(\log j)^{2\beta}}{j^{2\alpha}}.
   \end{equation}
   Taking the derivative of the function $f(x)= (\log x)^{2\beta} x^{-2\alpha}$ is easy to check that $f$ is 
decreasing for $x\geq e^{\beta/\alpha}$. Therefore, for $n\geq e^{\beta/\alpha}$, we can estimate the series (\ref{(1)}) 
against an integral and obtain
    \begin{equation} \label{(2)}
   a_{n+1} ^2 \leq B^2 \int_n ^\infty \frac{(\log x)^{2\beta}}{x^{2\alpha}} dx =
	B^2 n^{1-2\alpha }(\log n)^{2\beta}\int_1 ^\infty \Big( \frac{\log(nt)}{\log n}\Big)^{2\beta}\frac{dt}{t^{2\alpha}}.
   \end{equation}
   Since $1+x \leq e^x$, it follows that
\[
\Big(\frac{\log (nt)}{\log n} \Big)^{2\beta} = 
\Big( 1+ \frac{\log t}{\log n}\Big)^{2 \beta} \leq e^{\frac{\log t}{\log n}2\beta}=t^{\frac{2\beta}{\log n}}\,.
\]
If $2\beta/\log n \leq \alpha -1/2\,,$ i.e. $\log n \geq 2\beta/(\alpha-1/2)=4\beta/(2\alpha-1)$ or $n \geq e^{4\beta/(2\alpha-1)}\,,$ we get
\[
\int_1 ^\infty \Big ( \frac{\log(nt)}{\log n}\Big)^{2\beta} \frac{dt}{t^{2\alpha}} \leq \int_1 ^\infty t^{\alpha -1/2}\frac{dt}{t^{2\alpha}}
   =\int_1 ^\infty \frac{dt}{t^{\alpha +1/2}}= \frac{2}{2\alpha -1} \,.
\]
   Together with (\ref{(2)}), this yields the upper estimate for all $n\geq \max(N, e^{4\beta/(2\alpha-1)}).$
{\em  Step 2.} Now we turn to $(ii)$. As above we conclude
\begin{eqnarray*}   a_{n} ^2 & \ge &  A^2 \int_n ^\infty \frac{(\log x)^{2\beta}}{x^{2\alpha}} dx 
=A^2 n^{1-2\alpha }(\log n)^{2\beta}\int_1 ^\infty \Big( \frac{\log(nt)}{\log n}\Big)^{2\beta}\frac{dt}{t^{2\alpha}}
\\
& \ge & A^2 n^{1-2\alpha }(\log n)^{2\beta}\int_1 ^n \frac{dt}{t^{2\alpha}}
\\
& \ge & A^2 n^{1-2\alpha }(\log n)^{2\beta} \, \frac{1}{2 \alpha -1}\Big(1-\frac{1}{n^{2\alpha}} \Big)\, .
\end{eqnarray*}
The last factor is bounded below by $1/2$ if $n \ge 2$. This finishes the proof.
    \qed
\end{pf}

Equipped with this Proposition, it is now easy to transfer the two-sided estimates of 
$a_n (I_d:F_d(w) \to  L_2 (\mathbb{T}^d))$ that have been obtained in \cite{KSU1} into 
two-sided estimates of $a_n (I_d:F_d(w) \to  L_\infty (\mathbb{T}^d))$.
As an example we consider 
the case $r =2$\,. The following estimates have been shown in \cite[Theorem~4.15]{KSU1}.

\begin{prop}\label{Prop 2}
Let $s>0$ and  $d\in \mathbb{N}$. Then we have 
$$
    a_n (I_d: H^{s,2} (\mathbb{T}^d) \to  L_2 (\mathbb{T}^d))\leq \Big(\frac{32e}{d}\Big)^{s/2}n^{-s/d}
		\qquad \text{for } n\geq 9^d e^{d/2}
$$
and 
$$
   a_n (I_d: H^{s,2} (\mathbb{T}^d) \to  L_2 (\mathbb{T}^d)) \geq \Big(\frac{1}{e(d+2)}\Big)^{s/2}n^{-s/d}
	\qquad \text{for } n\geq 11^d e^{d/2}\,.
$$
\end{prop}

Together with Proposition \ref{Corollary}, this immediately implies

\begin{cor}\label{Cor 1}
Let $s>d/2$ and  $d\in \mathbb{N}$. Then we have for $n\geq 9^de^{d/2}$
$$
    a_{n+1} (I_d: H^{s,2} (\mathbb{T}^d) \to  L_\infty (\mathbb{T}^d))\leq 
\sqrt{\frac{2d}{2s -d}} \,  \Big(\frac{32e}{d}\Big)^{s/2}\, n^{1/2 - s/d}
$$
and for $n\geq 11^de^{d/2}$ it holds
$$
   a_n (I_d: H^{s,2} (\mathbb{T}^d) \to  L_\infty (\mathbb{T}^d)) \geq 
\sqrt{\frac{4d}{4s - d}} \, \Big(\frac{1}{e(d+2)}\Big)^{s/2}\, n^{1/2-s/d}\,.
$$
\end{cor}


\subsection{Embeddings of Sobolev spaces of dominating mixed smoothness}\label{Section 4.2}


Next we focus our attention on Sobolev spaces $H^{s,r} _{\rm{mix}}(\mathbb{T}^d)$ of 
dominating mixed smoothness. The weights are now
\[
w(k)=w_{s,r} ^{\rm{mix}}(k)=\prod_{j=1} ^d (1+|k_j|^r)^{s/r}
\]
where $0<s<\infty$ and $0<r \leq\infty$. 
It is known since a long time that
\[
H^{s,r} _{\rm{mix}}(\mathbb{T}^d) \hookrightarrow C(\mathbb{T}^d)
\qquad \Longleftrightarrow \qquad s>\frac 12\, .
\]
However, this can be calculated also directly by checking 
\begin{equation} \label{F5}
\sum_{k\in \mathbb{Z}^d} \frac{1}{w^{\rm{mix}} _{s,r} (k)^2} < \infty 
\qquad \Longleftrightarrow \qquad s>1/2 \,.
\end{equation}
By Theorem \ref{Theorem 2} this guarantees also $H^{s,r}_{\rm{mix}}(\mathbb{T}^d) \hookrightarrow \ca (\mathbb{T}^d)$.

The existence of the following limits was shown in \cite[Theorem 4.3, Corollaries 4.4 and 4.7]{KSU2} for the special 
values $r=1,2,2s$\,, 
but the proof works also for all other $0<r\le \infty$\,,
\begin{equation} \label{F6}
\lim _{n\rightarrow \infty} \frac{n^s a_n(I_d:H^{s,r}_{\rm{mix}} (\mathbb{T}^d)\to L_2 (\mathbb{T}^d))}{(\log n)^{(d-1)s}}=\Big [\frac{2^d}{(d-1)!}\Big ]^s \,.
\end{equation}
For approximation in $ L_\infty (\mathbb{T}^d)$ we have the following result.

\begin{thm}\label{Theorem 4}
Let $d\in \mathbb{N}$ , $s>1/2$ and $0<r\le \infty$. Then 
\[
\lim_{n \rightarrow \infty}\frac{ n^{s-1/2}\, a_n(I_d : H^{s,r} _{\rm{mix}}(\mathbb{T}^d) \to 
L_\infty (\mathbb{T}^d))}{(\log n)^{(d-1)s}} = 
\frac{1}{\sqrt{2s-1}}\Big [\frac{2^d}{(d-1)!}\Big ]^s \, .
\]
\end{thm}

\begin{pf}
We can apply the same arguments as in the proof of Theorem \ref{Theorem 3}, replacing Proposition \ref{Prop 1} by 
formula \eqref{F6}.
Here one has to take into account that
\[
\int_n ^\infty \frac{(\log x)^{(d-1)2s}}{x^{2s}} dx = n^{1-2s}(\log n)^{(d-1)2s} \, 
\int_1 ^\infty \Big( \frac{\log(nt)}{\log n}\Big)^{(d-1)2s}\frac{dt}{t^{2s}}
\]
see \eqref{(2)}, and 
\[
\lim_{n \to \infty} \int_1 ^\infty \Big( \frac{\log(nt)}{\log n}\Big)^{(d-1)2s}\frac{dt}{t^{2s}} = 
\int_1 ^\infty \frac{dt}{t^{2s}} = \frac{1}{2s-1}\, .
\]

\qed
\end{pf}

\begin{rmk}
 \rm (i) We can express this also in terms of a strong equivalence,
$$
a_n(I_d : H^{s,r} _{\rm{mix}}(\mathbb{T}^d) \to 
L_\infty (\mathbb{T}^d))\asymp \sqrt{\frac{n}{2s-1}}
a_n(I_d : H^{s,r} _{\rm{mix}}(\mathbb{T}^d) \to 
L_2 (\mathbb{T}^d))\,.
$$
Comparing $L_\infty$-approximation with $L_2$-approximation, the asymptotic rate is again worse by the factor $n^{1/2}$, 
as in the case of isotropic spaces. But now the value of the limit does not depend on the parameter $r$, and
the additional correction factor $1/\sqrt{2s-1}$ is independent of the dimension $d$, it depends only on the smoothness $s$.
We find this quite interesting.\\
(ii) As in case of the isotropic Sobolev spaces, 
for fixed $s,d$ and $r$ the asymptotic behaviour of the approximation numbers
$ a_{n} (I_d: \, H^{s,r}_{\rm{mix}} (\mathbb{T}^d) \to L_\infty (\mathbb{T}^d))$ as $n\to \infty$ has been 
known for more than 20 years,
see Temlyakov \cite{Te93}, but without explicit constants. The novelty of our results is that we can control 
the dependence of the constants on $s,d$ and $r$, so that we can even determine the optimal constants for $n \to \infty$.
\end{rmk}

In \cite{KSU2} several other $L_2$-estimates are given for approximation numbers $a_n$ with large $n$. 
Our general Theorem \ref{Theorem 2} in combination with Proposition \ref{Corollary} allows also to transfer 
each of these $L_2$-estimates into $L_\infty$-estimates. We conclude this subsection with another example which follows 
from \cite[Theorem 4.15]{KSU2}.

\begin{cor}\label{Cor 12}
Let $d\in \mathbb{N}$ and $s>1/2$. Then we have 
  $$
	a_{n} (I_d: \, H^{s,2}_{\rm{mix}} (\mathbb{T}^d) \to L_\infty (\mathbb{T}^d)) \le \sqrt{\frac{2}{2s -1}}
\left[\frac{(3\cdot \sqrt{2})^d}{(d-1)!}\right]^s\frac{(\ln n)^{(d-1)s}}{n^{s-1/2}}
$$
\hspace{4cm} if $ n > \max ( 27^d, e^{4(d-1)s/(2s-1)})$\,, and
$$
a_{n} (I_d: \, H^{s,2}_{\rm{mix}} (\mathbb{T}^d) \to L_\infty (\mathbb{T}^d)) \ge 
  \sqrt{\frac{1}{4s -2}}\,   \left[
 \frac{5}{6  \, d! (1+ \ln \sqrt{12})^d} \, \right]^s \, \frac{(\ln (2n))^{(d-1)s}}{n^{s-1/2}}
$$
\hspace{4cm} if $ n >  (12 \, e^2)^{d}$.
\end{cor}


\subsection{Embeddings of isotropic Besov spaces}\label{Section 4.3}


There is a rich literature on Besov spaces, we refer, e.g.,  to the monographs by
Peetre \cite{Pe} and Triebel \cite{Tr83}.
Special emphasis to the periodic situation is given in the monograph by Schmeisser and Triebel \cite{ST}.
As there, we shall also use the Fourier-analytic approach to introduce these spaces.
However, let us mention that they can be described also in terms of moduli of smoothness (differences)
or in terms of best approximation by trigonometric polynomials in the sense of equivalent norms. 
Clearly, for switching to these other characterizations we have to pay a price, 
expressed by a constant in the related inequalities, which depends in general on $d$. This will not be done here,
but it would be of certain interest.
\\
Later on we will be forced to deal with spaces with negative smoothness $s$, i.e., spaces of tempered distributions.
This requires some preparations.\\
Let $D(\mathbb{T}^d)$ denote the collection of all periodic infinitely differentiable complex-valued functions. 
In particular, $f(x)=f(y)$ if $x-y= k \in \mathbb{Z}^d$. The locally convex topology in $D(\mathbb{T}^d)$ is generated by the semi-norms
\[
\|\, f \, \|_\alpha := \sup_{x \in \mathbb{T}^d} \, |D^\alpha f(x)|\, , 
\]
where $\alpha \in \mathbb{N}^d_0$.
By $D'(\mathbb{T}^d)$ we denote the topological dual of $D(\mathbb{T}^d)$, i.e., the set of all linear functionals $g$
on $D(\mathbb{T}^d)$ such that 
\[
 |g(f)|\le c_N \, \sum_{|\alpha|\le N} \|\, f \, \|_\alpha 
\]
for all $f \in D(\mathbb{T}^d)$ and for some $N \in \mathbb{N}_0$ and $c_N>0$.
We equip $D' (\mathbb{T}^d)$ with the weak topology, i.e., 
\[
g= \lim_{j \to \infty} g_j \qquad \Longleftrightarrow \qquad g(f) = \lim_{j\to \infty} g_j(f)
\qquad \mbox{for all}\quad f \in D (\mathbb{T}^d)\, .
\]
The Fourier coefficients of  $g \in D' (\mathbb{T}^d)$ are defined by
\[
 \widehat{g} (k) := (2\pi)^{-d} \, g(e^{i kx})\, , \qquad k \in \mathbb{Z}^d\, .
\]
For every $g \in D' (\mathbb{T}^d)$ there are a constant $c_g>0$
and a natural number $K_g$  with
\[
|\widehat{g} (k) | \le c_g\, (1+|k|)^{K_g}\, , \qquad k \in \mathbb{Z}^d\, .
\]
Furthermore, every $g \in D'(\mathbb{T}^d)$ can be represented 
by its  Fourier series,
$$
g= \sum\limits_{k \in \mathbb{Z}^d} \widehat{g} (k)\, e^{i kx}\qquad\text{(convergence in } D'(\mathbb{T}^d)).
$$
For all these facts we refer to \cite[3.2.1]{ST}.
\\
Starting point of the Fourier-analytic approach to Besov spaces is a smooth dyadic decomposition of unity.
By $B_r $ we denote the Euclidean ball in $\R^d$ of radius $r$, centered at the origin.
Let $\psi \in C_0^\infty (\mathbb{R}^d)$ be a real-valued non-negative function such that
\begin{equation}\label{z-01}
\mbox{supp}\, \psi \subset B_{3/2}\, ,  \quad \psi (\xi/2)-\psi (\xi)\ge 0  \mbox{ for all } \xi\in\R^d\quad 
\mbox{and}\quad \psi (\xi)=1  \mbox{ for } \xi \in B_1\, .
\end{equation}
We define $\varphi (\xi) := \psi (\xi/2)-\psi (\xi) $,
\begin{equation}\label{z-02}
 \varphi_0 (\xi) := \psi (\xi)\,\qquad \mbox{and}\qquad \varphi_j (\xi) := \varphi (2^{-j+1}\xi)\, , \qquad j \in \mathbb{N}\, .
\end{equation}
Then
\begin{equation}\label{unity}
\sum_{j=0}^\infty \varphi_j (\xi) = 1 \qquad \mbox{for all}\quad \xi \in \mathbb{R}^d 
\end{equation}
and
\[
\mbox{supp}\, \varphi_j \subset B_{3\, \cdot 2^{j-1}}\setminus B_{2^{j-1}}\, ,  \qquad j \in \mathbb{N}\, .
\]
In addition we mention that in every point $x\in\R^d$ at most two of these functions $\varphi_j$, $j \in \mathbb{N}_0$, do not vanish.
For $g \in D'(\mathbb{T}^d)$ we define
\begin{equation}\label{g_j}
g_j (x) :=  \sum_{k \in \mathbb{Z}^d} \varphi_j (k)\, \widehat{g} (k)\, e^{i k x}\, .
\end{equation}
Since $\varphi_j$ has compact support, the $g_j$ are trigonometric polynomials.

\begin{defin}\label{Besov}
 Let $1\le p,q \le \infty$ and $s \in \mathbb{R}$. Then the periodic Besov space $B^s_{p,q}(\mathbb{T}^d)$ is the collection of all 
$g \in D'(\mathbb{T}^d)$ such that
\begin{equation}\label{ws-07}
\| \, g \, |B^s_{p,q}(\mathbb{T}^d)\|^\psi := 
\Big(\sum_{j=0}^\infty 2^{jsq} \, \Big\| \sum_{k \in \mathbb{Z}^d} \varphi_j (k)\, \widehat{g} (k)\, e^{i k x}\, 
\Big|L_p (\mathbb{T}^d)\Big\|^q
\Big)^{1/q} <\infty\, .
\end{equation}
\end{defin}

\begin{rmk}
 \rm
(i) Of course, the norm in \eqref{ws-07} depends on the generating function $\psi$ of the chosen 
smooth decomposition of unity $(\varphi_j)_j$. All these norms are equivalent.
If necessary, we indicate the dependence on $\psi$ by a superscript as in \eqref{ws-07}. Otherwise we simply write 
$\| \, g \, |B^s_{p,q}(\mathbb{T}^d)\|$.
\\
(ii)
The spaces $B^s_{p,q}(\mathbb{T}^d)$ are Banach spaces such that
\[
D(\mathbb{T}^d)
 \hookrightarrow B^s_{p,q}(\mathbb{T}^d)\hookrightarrow D'(\mathbb{T}^d)\, ,  
\]
see \cite[Theorem~3.5.1]{ST}. 
\\
(iii) 
Let us also mention that we have, in the sense of equivalent norms,
\[
H^{s,r} (\mathbb{T}^d) = B^s_{2,2}(\mathbb{T}^d)\,,
\]
where the equivalence constants depend on $\psi, s,d$ and $r$.
\end{rmk}

\begin{lem}\label{besov}
 We have the following chain of continuous embeddings 
\begin{equation}\label{ws-05}
B^{d/2}_{2,1}(\mathbb{T}^d) \hookrightarrow \ca (\mathbb{T}^d)\hookrightarrow B^0_{\infty,1} (\mathbb{T}^d )
\hookrightarrow C(\mathbb{T}^d)\, ,
\end{equation}
where the last two embedding operators have norm one.
\end{lem}

\begin{proof}
 The embedding $B^{d/2}_{2,1}(\mathbb{T}^d) \hookrightarrow \ca (\mathbb{T}^d)$ is well-known, we refer to the example 
 in Section 3.4.3 on page 141 in Triebel \cite{Tr77}.
Concerning the second embedding, by \eqref{unity} we have for $g \in \ca (\mathbb{T}^d)$ the inequality
\[
\sum_{j=0}^\infty  \Big\| \sum_{k \in \mathbb{Z}^d} \varphi_j (k)\, \widehat{g} (k)\, e^{i k x}\, 
\Big|L_\infty (\mathbb{T}^d)\Big\|
\le \sum_{j=0}^\infty  \, \sum_{k \in \mathbb{Z}^d} \varphi_j (k)\, |\widehat{g} (k)| \le \|\,  g\,  |\ca (\mathbb{T}^d)\|\, .
\]
Here we used the non-standard condition $\varphi_j \ge 0$, see \eqref{z-01}. 
Finally, we consider the embedding 
$ B^0_{\infty,1} (\mathbb{T}^d )\hookrightarrow C(\mathbb{T}^d)$.
For any $g \in B^0_{\infty,1}(\mathbb{T}^d)$, the continuous representative in the equivalence class of $g$ is just the Fourier series of $g$.
It follows
\begin{eqnarray*}
|g(x)| & = & \Big|\sum_{k \in \mathbb{Z}^d} \Big(\sum_{j=0}^\infty \varphi_j (k)\Big)\, \widehat{g} (k)\, e^{i k x}\, \Big|
\\
& \le &  \sum_{j=0}^\infty \Big\|\sum_{k \in \mathbb{Z}^d} \varphi_j (k)\, 
\widehat{g} (k)\, e^{i k x}\, \Big|L_\infty(\mathbb{T}^d)\Big\|
\le 
\| \, g \, |B^0_{\infty,1}(\mathbb{T}^d)\|\, .
\end{eqnarray*}
In both cases this proves that the norm of the corresponding embedding operator is $\le 1$.
By considering the constant function $g(x) = 1$, it is immediate that these norms are $\ge 1$ as well.
\end{proof}

\begin{rmk}
 \rm
(i) The embeddings stated in \eqref{ws-05} cannot be improved in the framework of periodic Besov spaces.
 Indeed, if we replace $B^{d/2}_{2,1}(\mathbb{T}^d)$ by $B^{d/2}_{2,q}(\mathbb{T}^d)$ for some $q>1$, then an 
embedding into $\ca(\T^d)$ cannot exist,  
 since $B^{d/2}_{2,q}(\mathbb{T}^d)$ contains unbounded functions, see \cite[3.4.3]{Tr77} and
 \cite{SiTr}.
 Also $B^0_{\infty,1} (\mathbb{T}^d )$ cannot be replaced by $B^0_{\infty,q} (\mathbb{T}^d )$
with $q<1$ in the embedding $\ca (\mathbb{T}^d)\hookrightarrow B^0_{\infty,1} (\mathbb{T}^d )$.
There are explicit counterexamples, we omit the details.
\\
(ii) We did not indicate the chosen norm in \eqref{ws-05}. The statement is true with respect to all
norms of the type $\|\, \cdot \, |B^0_{\infty,1}(\mathbb{T}^d)\|^\psi$, see \eqref{ws-07}.
 \end{rmk}

As an immediate consequence of Lemma \ref{besov} and Theorem \ref{Theorem 2} we obtain

\begin{cor}\label{Cor 6}
 Let $F_d(w)$ be given by a weight $w$ satisfying $\sum\limits_{k\in \Z^d} w(k)^{-2}<\infty$.
Then one has for all $n\in\mathbb{N}$ 
\begin{eqnarray*}
a_n(I_d: F_d(w)\to B_{\infty,1}^0 (\mathbb{T}^d))  &= & 
a_n(I_d: F_d(w)\to L_\infty (\mathbb{T}^d))\\
 & = & a_n(I_d: F_d(w)\to C(\mathbb{T}^d)) \\
& = &a_n(I_d: F_d(w)\to \ca (\mathbb{T}^d))\,.
\end{eqnarray*}
\end{cor}

\begin{rmk}
 \rm
(i) This result is  true for any norm on $B_{\infty,1}^0 (\mathbb{T}^d)$
of the type $\| \, \cdot  \, |B^0_{\infty,1}(\mathbb{T}^d)\|^\psi$, see Definition \ref{Besov}. 
\\
(ii) Often it is easier to calculate $a_n(I_d: F_d(w)\to B_{\infty,1}^0 (\mathbb{T}^d))$\,,
rather than $a_n(I_d: F_d(w)\to L_{\infty}(\mathbb{T}^d))$. 
So, also in this respect (and not only because of the control of the constants) this result may have some value. 
\end{rmk}


\subsection{Embeddings of Besov spaces of dominating mixed smoothness}\label{Section 4.4}


As mentioned in the Introduction, Sobolev and Besov spaces of dominating mixed smoothness may be helpful when 
dealing with approximation problems in high dimensions.
These spaces are much smaller than their isotropic counterparts.
The behaviour of the approximation numbers of embeddings of these classes into $L_p (\mathbb{T}^d)$
is closer to the one-dimensional case for embeddings of isotropic spaces than to 
the $d$-dimensional situation.
\\
Let $\psi\in C_0^\infty(\R)$ be a real-valued non-negative functions satisfying properties \eqref{z-01} in the case $d=1$.
The associated smooth dyadic decomposition of unity on $\mathbb{R}$ is denoted by $(\varphi_j)_{j=0}^\infty$, see \eqref{z-02}.
We shall work with tensor products of this decomposition. Observe that
\[
\sum_{\ell\in \mathbb{N}_0^d} \prod_{j =1}^d \varphi_{\ell_j} (x_j) = 1 \, \qquad 
\mbox{for all}\quad x \in \mathbb{R}^d\, .
\]

\begin{defin}\label{Dom}
Let $1\le p,q \le \infty$ and $t \in \mathbb{R}$. Then the periodic Besov space 
$S^t_{p,q}B(\mathbb{T}^d)$ of dominating mixed smoothness 
is the collection of all  $g \in D'(\mathbb{T}^d)$ such that the norm
\begin{equation}\label{ws-28}
\| \, g \, |S^t_{p,q}B(\mathbb{T}^d)\| := 
\Big(\sum_{\ell\in \mathbb{N}_0^d} 2^{|\ell|_1tq}
\Big\|\sum_{k\in \Z^d} \prod_{j =1}^d 
\varphi_{\ell_j} (k_j)\cdot\widehat{g}(k)e^{i k x}\big |L_p(\T^d)\Big\|^q\Big)^{1/q}
\end{equation}
is finite.
\end{defin}

\begin{rmk}
 \rm
(i) Of course, the norm in \eqref{ws-28} depends on the generating function $\psi$ of the chosen 
smooth decomposition of unity $(\varphi_j)_j$. Again all these norms are equivalent.
If necessary, we indicate the dependence on $\psi$ as in \eqref{ws-07} by
$\| \, g \, |S^t_{p,q}B(\mathbb{T}^d)\|^\psi$.
\\
(ii)
The spaces $S^t_{p,q}B(\mathbb{T}^d)$ are Banach spaces such that
\[
D(\mathbb{T}^d)
 \hookrightarrow S^t_{p,q}B(\mathbb{T}^d)\hookrightarrow D'(\mathbb{T}^d)\, .
\]
\\
(iii) 
Let us also mention that, in the sense of equivalent norms,
\[
H^{s,r}_{\rm{mix}} (\mathbb{T}^d) = S^s_{2,2}B(\mathbb{T}^d)\,,
\]
where the equivalence constants depend on $\psi, s, r$ and $d$.
\\
(iv) One of the most attractive features of these spaces is the following.
If 
\[
g (x) = \prod_{j=1}^d g_j (x_j) \, , \qquad x=(x_1, \ldots \, , d) \in \T^d\, , 
\]
then
\[
\| \, g \, |S^t_{p,q}B(\mathbb{T}^d)\|^\psi = \prod_{j=1}^d \| \, g_j \, |B^t_{p,q}(\mathbb{T})\|^\psi\,.
\]
\end{rmk}

\begin{lem}
 We have the chain of continuous embeddings 
\begin{equation}\label{ws-29}
\ca (\mathbb{T}^d)\hookrightarrow S^0_{\infty,1} B(\mathbb{T}^d )
\hookrightarrow C(\mathbb{T}^d)\, ,
\end{equation}
where the embedding operators are of norm one.
\end{lem}

\begin{proof}
For $g \in \ca (\mathbb{T}^d)$, using \eqref{unity}, we have
\begin{eqnarray*}
&& \hspace{-0.8cm} \sum_{\ell\in \mathbb{N}_0^d} 
\Big\| \sum_{k \in \mathbb{Z}^d} \Big(\prod_{j =1}^d
\varphi_{\ell_j} (k_j)\Big)\, \widehat{g} (k)\, e^{i k x}\, 
\Big|L_\infty (\mathbb{T}^d)\Big\|
\\
& \le &   
\, \sum_{k \in \mathbb{Z}^d} 
\Big[\sum_{\ell\in \mathbb{N}_0^d}
\Big(\prod_{j =1}^d \varphi_{\ell_j} (k_j)\Big)\Big] \, |\widehat{g} (k)| 
 \le   \|\,  g\,  |\ca (\mathbb{T}^d)\|\, .
\end{eqnarray*}
Next, we turn to the embedding 
$ S^0_{\infty,1}B (\mathbb{T}^d )\hookrightarrow C(\mathbb{T}^d)$.
It follows
\begin{eqnarray*}
|g(x)| & = & \Big|\sum_{k \in \mathbb{Z}^d} 
\Big[\sum_{\ell\in \mathbb{N}_0^d}
\Big(\prod_{j =1}^d \varphi_{\ell_j} (k_j)\Big)\Big]
\, \widehat{g} (k)\, e^{i k x}\, \Big|
\\
& \le &  
\sum_{\ell\in \mathbb{N}_0^d}
 \Big\|\sum_{k \in \mathbb{Z}^d} \Big(\prod_{j =1}^d \varphi_{\ell_j} (k_j)\Big)\,  
\widehat{g} (k)\, e^{i k x}\, \Big|L_\infty(\mathbb{T}^d)\Big\|
\\
&\le &
\| \, g \, |S^0_{\infty,1}B(\mathbb{T}^d)\|\, .
\end{eqnarray*}
In both cases this proves that the norm of the associated embedding operator is $\le 1$.
Considering $g(x) = 1$, it is immediate that these norms are $\ge 1$ as well.
\end{proof}

Similarly as for isotropic Besov spaces (see Corollary \ref{Cor 6}), we can derive the following 
consequence from Theorem \ref{Theorem 2}, where we can use any norm of the type
$\| \, g \, |S^0_{p,q}B(\mathbb{T}^d)\|^\psi $.

\begin{cor}\label{Cor 10}
 Let $F_d(w)$ be given by a weight $w$ satisfying $\sum\limits_{k\in \mathbb{Z}1d} w(k)^{-2}<\infty$.
Then one has for all $n\in\mathbb{N}$ 
\begin{eqnarray*}
a_n(I_d: F_d(w)\to S_{\infty,1}^0 B(\mathbb{T}^d)) & = &
a_n(I_d: F_d(w)\to L_\infty (\mathbb{T}^d))  
\\
& = & a_n(I_d: F_d(w)\to C (\mathbb{T}^d))
\\
& = & a_n(I_d: F_d(w)\to \ca (\mathbb{T}^d))
\,.
\end{eqnarray*}
\end{cor}

Next we deal with an example, where one has some advantage from the flexibility in the target space.
However, we will loose the explicit control of the $d$-dependence of the constants.
\\
We proceed by duality. Therefore we define
\[
\mathcal{B} (\mathbb{T}^d):= \Big\{ g \in D' (\mathbb{T}^d):  \|\, g\, |B(\mathbb{T}^d)\|:= 
\sup_{k \in \mathbb{Z}^d} |\widehat{g}(k)|<\infty \Big\}\, .
\]
Obviously we have, with equality of norms,  the duality
\[
\big(\ca (\mathbb{T}^d)\big)' = \mathcal{B} (\mathbb{T}^d)\, .
\]
It is a little bit
surprising that these quite natural spaces do not show up very often in mathematics.
Also we need the dual space of $C(\T^d)$. By the Riesz representation theorem we know
that $(C(\T^d))'$ is given by the collection of all signed $2\pi$-periodic  measures $\mu$ on $\T^d$
with norm given by the total variation $|\mu (\T^d)|$ of $\mu$.
We will use the notation ${\mathcal M} (\T^d)$ here for this Banach space.
Let $w=(w(k))_{k\in \Z^d}$ be a positive weight.
If $F_d(w)$ is defined as in \eqref{ws-06} then 
\[
\big(F_d(w)\big)' = F_d(1/w)\, .
\]
All these dual spaces are subspaces of $D'(\T^d)$ (since $D(\T^d)$ is dense in the original spaces).
Next we recall a result due to Hutton \cite{Hu}, see also \cite[Theorem~11.7]{P1} and \cite[Proposition~2.5.2]{CS}.
If $T: X \to Y$ is a compact linear operator between arbitrary Banach spaces, then
\[
a_n (T) = a_n (T')\, ,
\]
where $T':Y'\to X'$ denotes the dual operator of $T$. 

Combining this property with Theorems \ref{Theorem 1} and \ref{Theorem 2}
we get

\begin{cor}\label{Cor 3}
Let $F_d(w)$ be given by a weight $w$ satisfying $\sum\limits_{k\in \Z^d} w(k)^{-2}<\infty$.
Then one has for all $n\in\mathbb{N}$ 
\begin{eqnarray*}
a_n(I_d: \mathcal{M}(\mathbb{T}^d) \to F_d(1/w) ) &=&
a_n(I_d: \mathcal{B}(\mathbb{T}^d) \to F_d(1/w) ) 
\\
& = & a_n(I_d: F_d(w)\to \ca (\mathbb{T}^d) ) \,.
\end{eqnarray*}
\end{cor}

Before applying this result to Besov spaces of dominating mixed smoothness, we study the relation of $S^0_{1,\infty}B(\T^d)$
to $\mathcal{M}(\mathbb{T}^d)$ and $ \mathcal{B}(\mathbb{T}^d)$.

\begin{lem}\label{Lemma 4}
For all $g \in S^0_{1,\infty} B(\mathbb{T}^d)$ we have 
\[
|\widehat{g}(k)| \le 2^d \, \|\,  g \,|S^0_{1,\infty}B (\mathbb{T}^d)\|\, , \qquad k \in \mathbb{Z}^d\, ,
\] 
i.e., $S^0_{1,\infty}B (\mathbb{T}^d) \hookrightarrow \mathcal{B}(\mathbb{T}^d)$.
\end{lem}

\begin{proof}
For every $k\in \mathbb{Z}^d$ there exists an $ \ell=(\ell_1,\ldots,\ell_d) \in \mathbb{N}_0^d$ such that
\[
\prod_{j=1}^d \Big(\varphi_{\ell_j} (k_j) + \varphi_{\ell_{j}+1} (k_j)\Big) = 1 \, .
\] 
Using the abbreviation
\[
 g_\ell (x) := \sum_{k \in \mathbb{Z}^d} \Big(\prod_{j=1}^d \varphi_{\ell_j} (k_j)\Big) \, \widehat{g} (k)\, e^{i 2\pi kx}
\]
this implies, independently of the chosen decomposition of unity,
\begin{eqnarray*}
|\widehat{g}(k)| & = & \Big|\sum_{u \in \{0,1\}^d} \widehat{g}_{\ell +  u} (k)\Big|
\le  \sum_{u \in \{0,1\}^d} \Big| (2\pi)^{-d} \, \int_{\mathbb{T}^d} e^{-i kx} \,  g_{\ell + u} (x)\,  dx\Big| 
\\
 & \le &   \sum_{u \in \{0,1\}^d}  \| \, g_{\ell + u} \, |L_1 (\mathbb{T}^d)\|
\\
& \le &  2^d\,  \| \, g \, |S^0_{1,\infty} B(\mathbb{T}^d)\| \,. 
\end{eqnarray*}
\end{proof}

\begin{lem}\label{Lemma 10}
For all $\mu \in \mathcal{M}(\mathbb{T}^d)$ we have 
\[
\|\,  \mu \,|S^0_{1,\infty}B (\mathbb{T}^d)\| \le  \, \Big(\frac 1{2\pi}\,\max_{j=0,1} \|\, \mathcal{F}^{-1} \varphi_{j} \, |L_1 (\R)\|\Big)^d
\|\,  \mu \,| \mathcal{M}(\mathbb{T}^d)\|\, , 
\] 
i.e. there is a continuous embedding $ \mathcal{M}(\mathbb{T}^d)\hookrightarrow S^0_{1,\infty}B (\mathbb{T}^d)$.
\end{lem}

\begin{proof}
By $\mathcal{F}$ we denote the Fourier transform on $\R^d$, $\mathcal{F}^{-1}$ is the inverse transform.
It will be normalized as follows
\[
\mathcal{F} f (\xi) :=  \int_{\mathbb{R}^d}f(x)\, e^{- i x \xi}\, dx\, , \qquad \xi \in \mathbb{R}^d\, .
\]
Clearly, 
\[
\widehat{\mu} (k)= (2\pi)^{-d}\, \int_{\T^d} e^{ikx}\, d\mu\, , \qquad k \in \Z^d\, . 
\]
Let 
\[
\bar{\varphi}_{\bar{\ell}} (x) := \prod_{j =1}^d \varphi_{\ell_j} (x_j)\, , \qquad x=(x_1, \ldots, x_d)\, , \quad 
\bar{\ell}=(\ell_1, \ldots, \ell_d) \, .
\]
Because of 
\[
\sum_{k\in \Z^d} \bar{\varphi}_{\bar{\ell}} (k)\, \widehat{\mu}(k)e^{i k x}  = (2\pi)^{-d}\, 
\int_{\mathbb{T}^d} \Big(\sum_{k\in \Z^d} \bar{\varphi}_{\bar{\ell}} (k)\, e^{i k (x-y)}  \Big) \, d\mu (y)
\] 
and 
\[
\sum_{k\in \Z^d} \bar{\varphi}_{\bar{\ell}} (k)\, e^{i kz} = (2\pi)^d \, \sum_{m \in \Z^d} \mathcal{F}^{-1} \bar{\varphi}_{\bar{\ell}} (z + 2 \pi m)
\]
(Poisson's summation formula) we conclude
\begin{eqnarray*}
\Big\|\sum_{k\in \Z^d} && \hspace{-0.7cm} \bar{\varphi}_{\bar{\ell}} (k)\, \widehat{\mu}(k)e^{i k x} \, \Big|L_1 (\T^d)\Big\| 
\\
& = &
\Big\|\,  \int_{\mathbb{T}^d} \sum_{m \in \Z^d} \mathcal{F}^{-1} \bar{\varphi}_{\bar{\ell}} (x-y + 2 \pi m) \, d\mu(y)\, \Big|L_1 (\T^d)\Big\|
\\
& \le & (2\pi)^{-d} \, \sum_{m \in \Z^d}  \int_{\mathbb{T}^d} \int_{\mathbb{T}^d}  |\mathcal{F}^{-1} \bar{\varphi}_{\bar{\ell}} (x-y + 2 \pi m)| \, dx \, d\mu (y)
\\
& \le & (2\pi)^{-d} \,   \int_{\mathbb{T}^d} \int_{\R^d}  |\mathcal{F}^{-1} \bar{\varphi}_{\bar{\ell}} (x-y)| \, dx \, d\mu (y)
\\
& \le &  (2\pi)^{-d}  \|\, \mathcal{F}^{-1} \bar{\varphi}_{\bar{\ell}} \, |L_1 (\R^d)\|  \, |\mu (\T^d)|\, .
\end{eqnarray*}
The homogeneity of the Fourier transform yields
\[
\|\, \mathcal{F}^{-1} \bar{\varphi}_{\bar{\ell}} \, |L_1 (\R^d)\| \le 
\max \Big(\|\, \mathcal{F}^{-1} \varphi_{0} \, |L_1 (\R)\|, \|\, \mathcal{F}^{-1} \varphi_{1} \, |L_1 (\R)\| \Big)^d 
\]
independently of $\bar{\ell}$. This proves the claim.
\end{proof}

By using the abbreviation
\[
c_\psi := \frac 1{2\pi}\,\max (\|\, \mathcal{F}^{-1} \varphi_{0} \, |L_1 (\R)\|, \|\, \mathcal{F}^{-1} \varphi_{1} \, |L_1 (\R)\| )
\]
Lemma \ref{Lemma 4},  Lemma \ref{Lemma 10} and the so-called ideal property of approximation numbers  
yield the following.

\begin{lem}\label{Lemma 11}
Let $F_d(w)$ be given by a weight $w$ satisfying $\sum_{k\in \mathbb{Z}^d} w(k)^{-2}<\infty$.
Then one has for all $n\in\mathbb{N}$ 
\begin{eqnarray*}
c_\psi^{-d}\,  a_n(I_d: \mathcal{M}(\mathbb{T}^d) \to F_d(1/w) ) 
& \le &  a_{n} (I_d :  S^0_{1,\infty} B(\mathbb{T}^d) \to F_d (1/w)) 
\\
& \le & 
2^d \, a_n(I_d: \mathcal{B}(\mathbb{T}^d) \to F_d(1/w) )\, .  
\end{eqnarray*} 
\end{lem}

Let $\sigma \in \mathbb{R}$. For the dominating mixed case 
the standard lifting operator is given by 
\[
J^\sigma_{\rm{mix}} : ~g \to \sum_{k \in \mathbb{Z}^d} \widehat{g} (k)\, \prod_{j=1}^d (1+|k_j|^2)^{-\sigma/2}\, e^{i kx}\, .
\]
It is obvious that $J^\sigma_{\rm{mix}}$ is an isometry from $H^{s,2}_{\text{mix}}(\mathbb{T}^d)$ onto 
$H^{s+\sigma,2}_{\rm{mix}}(\mathbb{T}^d)$.

\begin{lem}\label{Lemma 5}
Let $\sigma \in \mathbb{R}$, $1 \le p,q \le \infty$ and $t \in \mathbb{R}$. Then  $J^\sigma_{\rm{mix}}$ 
maps $S^t_{p,q}B(\T^d)$ isomorphically onto $S^{t+\sigma}_{p,q}B(\T^d)$. 
\end{lem}

\begin{proof}
One can follow the arguments used in the proof of Theorem 2.3.8 in \cite{Tr83} which describes the isotropic situation.
\end{proof}

Now we turn to estimates of the approximation
numbers of the embeddings $I_d:  S^t_{1,\infty} (\mathbb{T}^d) \to L_2 (\mathbb{T}^d)$.
Essentially as a consequence of Corollary \ref{Cor 1}, Lemma \ref{Lemma 11} and Lemma  \ref{Lemma 5}   we find

\begin{cor} \label{Corollary 11}
Let $d\in \mathbb{N}$ and  $t>1/2$. 
Then there exists two positive constants $A,B$ such that
\[
 A\,  \frac{(\ln n)^{(d-1)t}}{n^{t-1/2}}\le 
a_{n} (I_d :  S^t_{1,\infty} B(\mathbb{T}^d) \to L_2 (\mathbb{T}^d))
 \leq   B\,  \frac{(\ln n)^{(d-1)t}}{n^{t-1/2}}
\]
holds for all $n \in \N$.
\end{cor}

\begin{proof}
{\em Step 1.} Estimate from above.
Observe that 
$(H^{t,2}_{\rm{mix}}(\mathbb{T}^d))' = H^{-t,2}_{\rm{mix}}(\mathbb{T}^d)$ with equality of norms.
We use the following commutative diagram.
\begin{center}
\setlength{\unitlength}{.7mm}
\begin{picture}(100,50)
\linethickness{.5pt} \put(-7,10){$S^0_{1,\infty} B(\mathbb{T}^d)$} \put(-7,40){$S^t_{1,\infty} B(\mathbb{T}^d)$}
\put(63,10){$H^{-t,2}_{\rm{mix}} (\mathbb{T}^d)$} \put(63,40){$L_2 (\mathbb{T}^d)$}

\put(18,42){\vector(1,0){44}} \put(12,35){\vector(0,-1){18}}
\put(19,12){\vector(1,0){43}} \put(67,17){\vector(0,1){20}}

\put(68,25){$J_{\rm{mix}}^t$} \put(40,45){$I_d$} \put(0,25){$J^{-t}_{\rm{mix}}$}
\put(40,15){$i_d$}

\end{picture}
\end{center}
Then the ideal property of  the approximation numbers yields
\begin{eqnarray*}
a_n(I_d: && \hspace*{-0.7cm} S^t_{1,\infty} B(\mathbb{T}^d) \to L_2(\mathbb{T}^d)) \le 
a_n (i_d : S^0_{1,\infty} B(\mathbb{T}^d) \to H^{-t,2}_{\rm{mix}}(\mathbb{T}^d)) 
\\
& \times  & \| J_{\rm{mix}}^r \, | H^{-t,2}_{\rm{mix}} (\mathbb{T}^d) \to L_2 (\mathbb{T}^d) \| \, 
\| J_{\rm{mix}}^{-t}\, | S^0_{1,\infty} B(\mathbb{T}^d) \to S^{-t}_{1,\infty} B(\mathbb{T}^d)  \|\,  
\\
& = & \, a_n (i_d : S^0_{1,\infty}B (\mathbb{T}^d) \to H^{-t,2}_{\rm{mix}}(\mathbb{T}^d))\, 
\| J_{\rm{mix}}^{-t}\, | S^0_{1,\infty} B(\mathbb{T}^d) \to S^{-t}_{1,\infty} B(\mathbb{T}^d)  \|  ,
\end{eqnarray*}
where we used Lemma \ref{Lemma 5} and $ \| J_{\rm{mix}}^t : H^{-t,2}_{\rm{mix}}(\mathbb{T}^d) \to L_2(\mathbb{T}^d) \| = 1$.
From Lemma \ref{Lemma 11} we conclude
\[
a_n (i_d : S^0_{1,\infty} B(\mathbb{T}^d) \to H^{-t,2}_{\rm{mix}}(\mathbb{T}^d))
\le 
2^d \, a_n (i_d : \mathcal{B}(\mathbb{T}^d) \to H^{-t,2}_{\rm{mix}}(\mathbb{T}^d))\, .
\]
Finally, applying 
Corollary \ref{Cor 3}, Corollary \ref{Cor 6} and Corollary \ref{Cor 12}, we obtain
\begin{eqnarray*}
a_n (i_d : \mathcal{B}(\mathbb{T}^d) \to H^{-t,2}(\mathbb{T}^d))  & = &  
a_n (i_d : H^{t,2}(\mathbb{T}^d) \to \ca (\mathbb{T}^d))
\\
&\le & C \, \frac{(\ln n)^{(d-1)t}}{n^{t-1/2}}\, , 
\end{eqnarray*}
where $C=C(d,t)$ is independent of $n$.\\
{\em Step 2.} Estimate from below.
Now we have the commutative diagram
\begin{center}
\setlength{\unitlength}{.7mm}
\begin{picture}(100,50)
\linethickness{.5pt} \put(-7,10){$S^t_{1,\infty} B(\mathbb{T}^d)$} \put(-7,40){$S^0_{1,\infty} B(\mathbb{T}^d)$}
\put(63,10){$L_{2} (\mathbb{T}^d)$} \put(63,40){$H^{-t,2}_{\rm{mix}} (\mathbb{T}^d)$}

\put(18,42){\vector(1,0){44}} \put(12,35){\vector(0,-1){18}}
\put(19,12){\vector(1,0){43}} \put(67,17){\vector(0,1){20}}

\put(68,25){$J_{\rm{mix}}^{-t}$} \put(40,45){$i_d$} \put(0,25){$J^{t}_{\rm{mix}}$}
\put(40,15){$I_d$}

\end{picture}
\end{center}
As in Step 1 we conclude
\begin{eqnarray*}   
a_n(i_d: && \hspace*{-0.7cm} S^0_{1,\infty} B(\mathbb{T}^d) \to H^{-t,2}_{\rm{mix}} (\mathbb{T}^d)) \le 
a_n (I_d : S^t_{1,\infty} B(\mathbb{T}^d) \to L_{2}(\mathbb{T}^d)) 
\\
& \times  & \| J_{\rm{mix}}^{-t} \, | L_2 (\mathbb{T}^d) \to H^{-t,2}_{\rm{mix}} (\mathbb{T}^d)  \| \, 
\| J_{\rm{mix}}^{t}\, | S^0_{1,\infty} B(\mathbb{T}^d) \to S^{t}_{1,\infty} B(\mathbb{T}^d)  \|\,  
\\
& = & \, 
a_n (I_d : S^t_{1,\infty} B(\mathbb{T}^d) \to L_{2}(\mathbb{T}^d)) 
\, \| J_{\rm{mix}}^{t}\, | S^0_{1,\infty} B(\mathbb{T}^d) \to S^{t}_{1,\infty} B(\mathbb{T}^d)\| \, .
\end{eqnarray*}
The estimate from below is completed by taking into account Lemma \ref{Lemma 11},
Corollary \ref{Cor 3}, Corollary \ref{Cor 6} and Corollary \ref{Cor 12}.
\end{proof}

\begin{rmk}
 \rm
(i) Corollary \ref{Corollary 11}  seems to be partly a novelty.
Our results remove the technical restriction $r>1$ which appeared 
in the previous contributions by Temlyakov \cite{Tem86} and Romanyuk \cite{Rom1}, 
see also Romanyuk \cite{Rom2}.
Notice that $r>1/2$ is the optimal restriction since
\[
S^r_{1,\infty} B(\mathbb{T}^d) \hookrightarrow  L_2 (\mathbb{T}^d) \qquad \Longleftrightarrow 
\qquad r > 1/2\, .
\]
In fact, Temlyakov \cite{Tem86} and Romanyuk \cite{Rom1} investigated  Kolmogorov numbers of the operator
$I_d :  S^r_{1,\infty} B(\mathbb{T}^d) \to L_2 (\mathbb{T}^d)$, 
but for operators mapping a Banach space into a \emph{Hilbert space}, 
approximation numbers and Kolmogorov numbers coincide, see \cite[Proposition~11.6]{P1}. 
\\
(ii) 
Of course, $A$ and $B$ depend on $d$. In the proof of Corollary \ref{Corollary 11} we have full control about the behaviour of the constants 
except when dealing with the lifting operator.
Of course, one can give estimates of    
\[
\| J_{\rm{mix}}^{-t}\, | S^0_{1,\infty} B(\mathbb{T}^d) \to S^{-t}_{1,\infty} B(\mathbb{T}^d)  \| \qquad \mbox{and}
\qquad \| J_{\rm{mix}}^{t}\, | S^0_{1,\infty} B(\mathbb{T}^d) \to S^{t}_{1,\infty} B(\mathbb{T}^d) \|
\]
in dependence on $d$. However, it is not clear whether this method results in good constants.
So we omit details.
\end{rmk}


\subsection{Approximation in the $L_p$-norm}\label{Section 4.5}


Since we are able to control the behaviour of the approximation numbers 
$a_n(I_d: F_d(w)\to L_\infty (\mathbb{T}^d))$
and $a_n(I_d: F_d(w)\to L_2 (\mathbb{T}^d))$, it is quite natural to ask if this can also be done for
approximation in the $L_p$-norm, $2 < p< \infty$.
Accordingly, we finish the paper with some upper estimates in the $L_p$-norm.

\begin{prop}\label{Prop 8}
Let $2<p<\infty$, define $r$ by $1/r=1/2-1/p$, assume that $F_d(w)$ is given by weights $w(k)$ satisfying 
$\sum_{k\in\Z^d} w(k)^{-r}<\infty$. Moreover,
let $(\sigma_j)_{j\in\mathbb{N}}$ denote the non-increasing rearrangement of $(1/w(k))_{k\in\mathbb{Z}^d}$. 
Then one has for all $n\in\N$ 
\[
a_n(I_d: F_d(w)\to L_p (\mathbb{T}^d)) \le \Big( \sum_{j=n}^\infty \sigma_j^r \Big)^{1/r}\,.
\]
\end{prop}

\begin{proof}
Let the operators $A,B$ and $D$ be defined as in the proof of Theorem \ref{Theorem 1}.
We consider the following commutative diagram.

\begin{center}
\setlength{\unitlength}{.7mm}
\begin{picture}(100,50)
\linethickness{.5pt} \put(5,10){$\ell_2 (\mathbb{Z}^d)$} \put(3,40){$F_d(w)$}
\put(63,10){$\ell_{p'} (\mathbb{Z}^d)$} \put(63,40){$L_p (\mathbb{T}^d)$}

\put(18,42){\vector(1,0){44}} \put(12,35){\vector(0,-1){18}}
\put(19,12){\vector(1,0){43}} \put(67,17){\vector(0,1){20}}

\put(68,25){$B$} \put(40,45){$I_d$} \put(5,25){$A$}
\put(40,15){$D$}

\end{picture}
\end{center}
To make sure that this diagram really makes sense, we have to check the mapping properties of the 
operators $B$ and $D$. Concerning $B$, we know already that
$$
\|B:\ell_2(\mathbb{Z}^d)\to L_2(\mathbb{T}^d)\|=\|B:\ell_1(\mathbb{Z}^d)\to L_\infty(\mathbb{T}^d)\|=1\,.
$$
Complex interpolation of $L_p$-spaces (see e.g. \cite[section 5.1]{BL} or \cite[subsection 1.18.4]{TrI}) yields 
$$
\|B: \ell_{p'}(\mathbb{Z}^d) \to  L_p(\mathbb{T}^d) \| \le 1\,,
$$ 
where the conjugate index $p'$ of $p$ is defined by $\frac 1p+\frac 1{p'}=1$.

Now we turn to $D$. Since $1/2+1/r=1-1/p=1/{p'}$ and $\big(1/w(k)\big)_{k\in\Z^d}\in\ell_r(\Z^d)$, H\"older's 
inequality implies that $D$ maps $\ell_2(\Z^d)$ into $\ell_{p'}(\Z^d)$. 
Moreover we have
\[
a_n(D: \ell_2 (\mathbb{Z}^d) \to \ell_{p'} (\mathbb{Z}^d))= \Big( \sum_{j=n}^\infty \sigma_j^r \Big)^{1/r}
\]
where $1/r = 1/p'-1/2 = 1/2 - 1/p$, see \cite[Theorem~11.11]{P1}. By the multiplicativity of the 
approximation numbers, the desired estimate follows.
\end{proof}

The following technical lemma can be proved parallel to Proposition \ref{Corollary}.

\begin{lem}\label{Corollaryb}
Let $2 < p < \infty$ and put $1/r = 1/2 - 1/p$, let $\alpha > 1/r$, $C>0$ and $N\in \mathbb{N}$. Assume that 
\[
a_n (I_d:F_d(w) \to  L_2 (\mathbb{T}^d))\leq C\, n^{-\alpha}\quad \text{for all } \, n \geq N.
\]
Then it follows that
\[
  a_{n+1}(I_d:F_d(w) \to  L_p (\mathbb{T}^d)) \leq
       \frac{C}{(\alpha r-1)^{1/r}} \, n^{1/r - \alpha}\quad \text{for all } \, n \geq N.
\]
\end{lem}

Combining Lemma \ref{Corollaryb}, Proposition \ref{Prop 8} and Proposition \ref{Prop 2}
we can supplement Corollary \ref{Cor 1} as follows.

\begin{cor}\label{Cor 12b}
Let  $d\in \mathbb{N}$, $2 < p< \infty$ and $s>d\big(\frac 12 - \frac 1p\big)$.
Then we have for 
$n\geq 9^d e^{d/2}$ the estimate
$$
a_{n+1} (I_d: H^{s,2} (\mathbb{T}^d) \to  L_p(\mathbb{T}^d))
\leq
\frac{1}{(\alpha r-1 )^{1/r}}
 \Big(\frac{32e}{d}\Big)^{s/2}\, n^{1/2 - 1/p - s/d}
$$
where $1/r = 1/2 - 1/p$ and $\alpha=s/d$.
\end{cor}

\begin{rmk}
 \rm
(i)
Note that the condition $s>d\big(\frac 12 - \frac 1p\big)$ is no restriction, since
\[
H^{s,2}(\mathbb{T}^d) \hookrightarrow  L_p (\mathbb{T}^d) \qquad \Longleftrightarrow 
\qquad \frac sd>\frac 12 - \frac 1p\, .
\]
(ii) The upper bound gives the correct asymptotic rate in $n$ (which is well-known),
the novelty is the explicit constant.\\
(iii) Corollary \ref{Cor 12b} is just one example of possible applications of Proposition \ref{Prop 8}. 
One can also derive results involving Sobolev spaces of dominating mixed smoothness. 
This could easily be done by introducing an additional log-term in the assumption of the 
technical Lemma \ref{Corollaryb} and combining this with known estimates 
of $a_n(I_d: H^s_{\rm{mix}}(\T^d)\to L_2(\T^d))$ from \cite{KSU2}.
\end{rmk}


\end{document}